\documentclass[12pt]{amsart}

\usepackage{amsmath,amssymb,epic,lscape}

\newtheorem{theorem}{Theorem}[section]
\newtheorem{proposition}[theorem]{Proposition}

\newtheorem{conjecture}[theorem]{Conjecture}

\theoremstyle{definition}

\theoremstyle{remark}
\newtheorem{remark}[theorem]{Remark}

\numberwithin{equation}{section}
\providecommand{\bysame}{\leavevmode\hbox to3em{\hrulefill}\thinspace}

\def\DJ{{\hbox{D\kern-.8em\raise.15ex\hbox{--}\kern.35em}}}
\def\DJo{$\;$\kern-.4em
    \hbox{D\kern-.8em\raise.15ex\hbox{--}\kern.35em okovi\'c}}

\def\NSERC{Supported in part by an NSERC Discovery Grant.}

\def\impl{\ \Rightarrow\ }

\def\bZ{{\mbox{\bf Z}}}

\renewcommand{\subjclassname}{\textup{2000} Mathematics Subject
Classification }

\begin{document}

\title[Near-normal sequences]
{Some new near-normal sequences}

\author[D.\v{Z}. \DJ okovi\'{c}]
{Dragomir \v{Z}. \DJ okovi\'{c}}

\address{Department of Pure Mathematics, University of Waterloo,
Waterloo, Ontario, N2L 3G1, Canada}

\email{djokovic@uwaterloo.ca}

\thanks{\NSERC}


\date{}

\begin{abstract}
The normal sequences $NS(n)$ and near-normal sequences $NN(n)$ play an
important role in the construction of orthogonal designs and Hadamard 
matrices. They can be identified with certain base sequences 
$(A;B;C;D)$, where $A$ and $B$ have length $n+1$ and $C$ and $D$ 
length $n$. C.H. Yang conjectured that near-normal sequences exist 
for all even $n$. While this has been confirmed for $n\le30$, so far 
nothing else was known for $n>30$. Our main result is that $NN(32)$ 
consists of 8 equivalence classes and we exhibit their representatives. 
We also construct representatives for two equivalence classes of 
$NN(34)$. On the other hand, we have shown by exhaustive computer 
searches that $NS(31)$ and $NS(33)$ are void. 
\end{abstract}

\maketitle
\subjclassname{ 05B20, 05B30 }
\vskip5mm

\section{Introduction}

We deal with quadruples $(A;B;C;D)$ of binary sequences, i.e.,
sequences with entries $\pm1$. Base sequences are such quadruples,
with $A$ and $B$ of length $m$ and $C$ and $D$ of length $n$,
such that the sum of their nonperiodic autocorrelation
functions is a $\delta$-function. The collection of such
sequences is denoted by $BS(m,n)$.

In section \ref{baza} we recall the definition of base 
sequences and how they can be used to construct Hadamard matrices.
In section \ref{nor-niz} we define normal sequences, $NS(n)$, and
near-normal sequences, $NN(n)$, as some special classes of base
sequences $BS(n+1,n)$. We also recall some basic facts
about these sequences: their use to construct Yang multiplications,
their importance for the construction of $T$-sequences, and
what is known about their existence. The $T$-sequences are
not binary; they are quadruples of ternary sequences with entries
from $\{0,\pm1\}$, all of the same length. See the main text for 
the complete definition.

In section \ref{glavni} we describe the new results that we have
obtained. We have classified the near-normal sequences $NN(32)$
and constructed two non-equivalent near-normal sequences 
in $NN(34)$. Let us mention that in our recent paper \cite{DZ2} 
we have classified the near-normal sequences $NN(n)$ for all even 
$n\le30$. We have introduced there two equivalence relations in 
$NN(n)$: $BS$- and $NN$-equivalence. In this note we use only
the $NN$-equivalence. We also report that our exhaustive
searches have shown that $NS(31)=NS(33)=\emptyset$.
As a consequence of these facts, we deduce that $63$ and $67$
are not Yang numbers while $69$ is such a number.
By definition, a Yang number is an odd integer $2s+1$ such
that $NS(s)$ or $NN(s)$ is not empty.

\section{Base sequences} \label{baza}

We denote finite sequences of integers by capital letters. If, say,
$A$ is such a sequence of length $n$ then we denote its elements
by the corresponding lower case letters. Thus
$$ A=a_1,a_2,\ldots,a_n. $$
To this sequence we associate the polynomial
$$ A(x)=a_1+a_2x+\cdots+a_nx^{n-1} , $$
which we view as an element of the Laurent polynomial ring
$\bZ[x,x^{-1}]$. (As usual, $\bZ$ denotes the ring of integers.)
The nonperiodic autocorrelation function $N_A$ of $A$ is defined by:
$$ N_A(i)=\sum_{j\in\bZ} a_ja_{i+j},\quad i\in\bZ, $$
where $a_k=0$ for $k<1$ and for $k>n$. Note that
$N_A(-i)=N_A(i)$ for all $i\in\bZ$ and $N_A(i)=0$ for $i\ge n$.
The norm of $A$ is the Laurent polynomial $N(A)=A(x)A(x^{-1})$.
We have
$$ N(A)=\sum_{i\in\bZ} N_A(i) x^i . $$
To the sequence $A$ we associate two other sequences of the same
length: the negation
$$ -A = -a_1,-a_2,\ldots,-a_n $$
and the alternation
$$ A^* = a_1,-a_2,a_3,-a_4,\ldots,(-1)^{n-1} a_n. $$
By $A,B$ we denote the concatenation of the sequences $A$ and $B$.

The {\em base sequences} consist of four 
$\{\pm1\}$-sequences $(A;B;C;D)$, with $A$ and $B$ of length $m$ 
and $C$ and $D$ of length $n$, such that 
\begin{equation} \label{norma}
N(A)+N(B)+N(C)+N(D)=2(m+n).
\end{equation}
We denote by $BS(m,n)$ the set of such base sequences with
$m$ and $n$ fixed. 

It is known that $BS(n+1,n)\ne\emptyset$ for
$0\le n\le35$ (see \cite{KSY,KS}) and that $BS(2n-1,n)\ne\emptyset$ 
for all even $n=2,4,\ldots,36$ (see \cite{KSY,KS,KT}).

Base sequences can be used to construct Hadamard matrices. 
Recall that a Hadamard matrix of order $m$ is a $\{\pm1\}$-matrix
$H$ of order $m$ such that $HH^T=mI_m$, where
$T$ denotes the transpose and $I_m$ the identity matrix.
For instance if $(A;B;C;D)\in BS(n,n)$ then we can construct a
Hadamard matrix $H$ of order $4n$ as follows. Let $A^c$ denote
the circulant matrix having $A$ as its first row, and define
similarly the circulants $B^c,C^c$ and $D^c$. 
The construction of $H$ is based on the Goethals--Seidel array
$$
\left[ \begin{array}{rrrr}
		U & XR & YR & ZR \\
		-XR & U & -Z^TR & Y^TR \\
		-YR & Z^TR & U & -X^TR \\
		-ZR & -Y^TR & X^TR & U
\end{array} \right]. $$
To obtain $H$ we just substitute the symbol $R$ with the $n\times n$ 
matrix having ones on the back-diagonal and all other entries zero,
and substitute (in any order) the symbols $U,X,Y,Z$ with the four circulants
$A^c,B^c,C^c,D^c$. The condition (\ref{norma}) guarantees that
$H$ is indeed a Hadamard matrix.

In connection with this construction, observe that there is a map
$BS(m,n)\to BS(m+n,m+n)$ sending
$$ (A;B;C;D)\to(A,C;\, A,-C;\, B,D;\, B,-D). $$

\section{Normal, near-normal and $T$-sequences} \label{nor-niz}

{\em Normal\;} resp. {\em near-normal sequences}, originally defined 
by C.H. Yang \cite{Y}, can be viewed as a special type of base 
sequences $BS(n+1,n)$ (see \cite{KSY,DZ1}), namely such that 
$b_i=a_i$ resp. $b_i=(-1)^{i-1} a_i$ for $1\le i\le n$. 
We denote by $NS(n)$ resp. $NN(n)$ the subset
of $BS(n+1,n)$ consisting of normal resp. near-normal sequences.

Very little is known about the existence of normal sequences $NS(n)$.
{\em Golay sequences} of length $n$ are two $\{\pm1\}$-sequences $(A;B)$ of 
length $n$ such that $N(A)+N(B)=2n$. If such sequences of length $n$
exist, we say that $n$ is a {\em Golay number}. The known Golay numbers
are $n=2^a10^b26^c$, where $a,b,c$ are arbitrary nonnegative integers.
If $n$ is a Golay number, then $NS(n)\ne\emptyset$. Indeed, if
$(A;B)$ are Golay sequences of length $n$, then
$(A,+;A,-;B;B)\in NS(n)$. For $n\le 30$ it 
is known (see \cite{DZ1,HCD}) that $NS(n)=\emptyset$ iff
$$ n\in \{6,14,17,21,22,23,24,27,28,30\}. $$

The case of near-normal sequences $NN(n)$ is apparently more
promising. We mention that if $n>1$ and $NN(n)\ne\emptyset$, then
$n$ must be even. The following question (now known as Yang's
conjecture) was raised about twenty years ago.

\begin{conjecture} (Yang \cite{Y}) \label{YC}
$NN(n)\ne\emptyset$ for all positive even $n$'s.
\end{conjecture}

It has been known since 1994 that near-normal
sequences exist for even $n\le30$ (see \cite{KSY}), but 
nothing else was known for larger values of $n$ (see \cite{HCD}). 

Some of the most powerful methods for constructing orthogonal
designs and Hadamard matrices are based on $T$-sequences
(see \cite{HCD}). Let us recall that $T$-{\em sequences} are 
quadruples $(A;B;C;D)$ of $\{0,\pm1\}$-sequences of the 
same length $n$ such that $N(A)+N(B)+N(C)+N(D)=n$ and, 
for each $i$, exactly one of $a_i,b_i,c_i,d_i$ is nonzero. 
We denote by $TS(n)$ the set of $T$-sequences of length $n$. 
It is known that $TS(n)\ne\emptyset$ for all
odd $n<100$ different from $73,79$ and $97$. It has been
conjectured that $TS(n)\ne\emptyset$ for all odd integers $n$.

Normal and near-normal sequences are important for the
construction of $T$-sequences.
If $NN(s)$ and $BS(m,n)$ are nonempty then there
is a map, called {\em Yang multiplication} \cite{Y,KSY}
\begin{equation} \label{YM}
NN(s) \times BS(m,n) \to TS((2s+1)(m+n)),
\end{equation}
and a similar statement is valid for $NS(s)$. For that
reason it is customary to refer to the odd integer $2s+1$
as a {\em Yang number} if $NS(s)$ or $NN(s)$ is nonempty.

\begin{remark} \label{greske}
While implementing in Maple the Theorems 1-4 of \cite{Y} we
discovered two misprints:
In the definition of $\tau_k$ on p. 770 one should replace the two
$f_k^*$'s with $f_k$'s, and in the definition of $\beta_k$ on 
p. 773 one should replace $A$ with $A^*$. 
The asterisk is used in \cite{Y}, and in this remark, to denote 
the reversed sequence. These errors were not easy to locate and 
correct, the same errors appear in \cite{KJS}.
\end{remark}

It is well known that there are infinitely many Yang numbers.
Indeed, if $s$ is a Golay number then $NS(2s+1)\ne\emptyset$,
and so $2s+1$ is a Yang number.
The known Yang numbers up to 100 are;
$$ 1,3,5,\ldots,31,33,37,39,41,45,49,51,53,57,59,61,65,81. $$
It is also known that $35,43,47$ and $55$ are not Yang numbers.

\section{New results} \label{glavni}

We show first that near-normal sequences $NN(32)$
and $NN(34)$ exist, and thereby confirm Yang's conjecture
for $n=32$ and $n=34$. These sequences have been discovered by 
using the same algorithm as in our paper \cite{DZ1}.

\begin{proposition} \label{skoro}
The sets $NN(32)$ and $NN(34)$ are nonempty.
\end{proposition}

\begin{proof}
To prove this, it suffices to verify that
$(A;B;C;D)\in NN(32)$, where
\begin{eqnarray*}
A&=&+,-,+,-,+,-,-,-,+,-,-,-,-,-,+,+,-,+,+,-, \\
&\ &-,+,+,-,+,-,-,+,-,+,-,-,+; \\
B&=&+,+,+,+,+,+,-,+,+,+,-,+,-,+,+,-,-,-,+,+, \\
&\ &-,-,+,+,+,+,-,-,-,-,-,+,-; \\
C&=&+,+,+,+,-,-,-,-,+,+,+,-,+,-,-,+,+,+,-,+, \\
&\ &+,+,-,-,+,+,-,+,+,+,-,+; \\
D&=&+,+,+,+,-,+,+,-,+,-,-,+,-,-,+,+,+,-,-,-, \\
&\ &-,-,+,-,+,-,+,+,+,+,-,+,
\end{eqnarray*}
and $(P;Q;R;S)\in NN(34)$, where
\begin{eqnarray*}
P&=&+,-,+,+,+,-,+,+,+,-,+,+,-,+,+,+,+,-,-,-, \\
&\ &+,+,+,+,+,+,-,-,-,-,+,-,-,-,+; \\
Q&=&+,+,+,-,+,+,+,-,+,+,+,-,-,-,+,-,+,+,-,+, \\
&\ &+,-,+,-,+,-,-,+,-,+,+,+,-,+,-; \\
R&=&+,+,-,-,-,-,+,+,-,-,+,-,+,+,-,-,-,-,+,+, \\
&\ &-,-,+,+,-,+,+,-,+,-,+,-,-,+; \\
S&=&+,+,+,-,-,+,-,+,-,+,-,-,-,-,-,-,+,+,+,+, \\
&\ &+,+,+,-,+,+,+,+,-,-,+,+,-,+.
\end{eqnarray*}
The signs ``$+$'' and ``$-$'' stand for $+1$ and $-1$, respectively. 
It is tedious to verify by hand that $(A;B;C;D)$ and $(P;Q;R;S)$ are
base sequences, but this can be easily done on a computer. 
The additional requirements for near-normality can be checked by 
inspection.
\end{proof}

\begin{proposition} \label{posledica}
The number $69$ is a (new) Yang number.
The numbers $63$ and $67$ are not Yang numbers.
\end{proposition}

\begin{proof} The first assertion holds since $NN(34)\ne\emptyset$. 
Our exhaustive computer searches showed that $NS(31)=NS(33)=\emptyset$.
This implies the second assertion.
\end{proof}

As 32 is a Golay number, we know that $NS(32)\ne\emptyset$. 
Hence, the first unresolved case for the existence question of
normal sequences $NS(n)$ is now $n=34$.

In our paper \cite{DZ2}, we have introduced two equivalence relations  
for near-normal sequences $NN(n)$: The $BS$-equivalence and the
$NN$-equivalence. The former is finer than the latter. An 
$NN$-equivalence class may contain 1,2 or 4 $BS$-equivalence classes.
In this note we use only the $NN$-equivalence.

In the case $n=32$ we have carried out an exhaustive search and
found that $NN(32)$ consists of 8 $NN$-equivalence classes.
In the case $n=34$ our search was not complete and we constructed 
only two non-equivalent near-normal sequences. 
We list in Table 1 the representatives of these 10 
$NN$-equivalence classes. The representatives are written in the 
compact encoded form. 
For the description of our encoding scheme see \cite{DZ1,DZ2}. 
The sequences $(A;B;C;D)$ and $(P;Q;R;S)$ displayed above
are the first sequences in Table 1 for $n=32$ and $n=34$, respectively.
The numbers $a,b,c,d$ resp. $a^*,b^*,c^*,d^*$ are the sums of the 
corresponding sequences $A,B,C,D$ resp. $A^*,B^*,C^*,D^*$.
Note that
$$ (A;B;C;D)\in NN(n) \impl (A^*;B^*;C^*;D^*) \in NN(n). $$

\begin{center}
\begin{tabular}{|r|l|l|r|r|}
\multicolumn{5}{c}{Table 1: Near-normal sequences $NN(n)$} \\ \hline 
\multicolumn{1}{|c|}{} & \multicolumn{1}{c|}{$A$ \& $B$} & 
\multicolumn{1}{c|}{$C$ \& $D$} & \multicolumn{1}{c|}{$a,b,c,d$} & 
\multicolumn{1}{c|}{$a^*,b^*,c^*,d^*$} \\ \hline
\multicolumn{5}{c}{$n=32$} \\ \hline
1 & $07656587173587123$ & $1611375364252851$ & $-5,5,8,4$ & $7,-7,-4,4$ \\
2 & $07643217328262853$ & $1657222254564485$ & $3,-7,6,-6$ & $-5,1,-10,2$ \\
3 & $07841512343414140$ & $1663752642548557$ & $9,7,0,0$ & $9,7,0,0$ \\
4 & $07651732153537650$ & $1767258654155337$ & $3,9,-2,6$ & $11,1,-2,2$ \\
5 & $07156434121787153$ & $1867665578785216$ & $3,9,-6,2$ & $11,1,-2,2$ \\
6 & $05671462321465123$ & $1166547238573585$ & $11,1,-2,2$ & $3,9,-6,-2$ \\
7 & $05128282658784653$ & $1653815347277422$ & $-1,-11,2,2$ & $-9,-3,-2,6$ \\
8 & $05126417143285123$ & $1657686527418862$ & $11,1,-2,-2$ & $3,9,6,2$ \\
\hline
\multicolumn{5}{c}{$n=34$} \\ \hline
1 & 076417646512321462 & 16738541372344337 & $7,7,-2,6$ & $9,5,4,-4$ \\
2 & 076782178767646231 & 17621532262576812 & $-5,3,10,-2$ & $5,-7,0,-8$ \\
\hline 
\end{tabular} \\
\end{center}

The $NN(32)$ above show that $65$ is a Yang number. However this 
fact is already known since $NS(32)\ne\emptyset$. Nevertheless, 
each of the above ten near-normal sequences provides infinitely many 
(probably new) Hadamard matrices by using the Yang multiplication (\ref{YM}) 
and the infinite supply of known Williamson-type matrices 
(see e.g. \cite{SY}).


\begin{thebibliography}{99}

\bibitem{HCD}
C.J. Colbourn and J.H. Dinitz, Editors, Handbook of Combinatorial Designs,
2nd edition, Chapman \& Hall, Boca Raton/London/New York, 2007.

\bibitem{DZ1}
D.\v{Z}. \DJo{},
Aperiodic complementary quadruples of binary sequences,
JCMCC {\bf 27} (1998), 3--31. Correction: ibid {\bf 30} (1999), p. 254.

\bibitem{DZ2}
\bysame, Classification of near-normal sequences, 
Discrete Mathematics, Algorithms and Applications, 
{\bf 1}, No. 3 (2009), 389--399.

\bibitem{KT}
H. Kharaghani and B. Tayfeh-Rezaie, A Hadamard matrix of order 428,
J. Combin. Designs {\bf 13} (2005), 435--440.

\bibitem{KSY}
C. Koukouvinos, S. Kounias, J. Seberry, C.H. Yang and J. Yang,
Multiplication of sequences with zero autocorrelation,
Austral. J. Combin. {\bf 10} (1994), 5--15.

\bibitem{KJS}
C. Koukouvinos and J. Seberry, Addendum to further results on base
sequences, disjoint complementary sequences, $OD(4t;t,t,t,t)$, and 
the excess of Hadamard matrices,
Congressus Numerantium {\bf 82} (1991), 97--103.

\bibitem{KS}
S. Kounias and K. Sotirakoglu, Construction of orthogonal sequences,
Proc. 14-th Greek Stat. Conf. 2001, 229--236 (in Greek).

\bibitem{SY}
J. Seberry and M. Yamada, Hadamard matrices, sequences and block
designs, in Contemporary Design Theory: A Collection of Surveys,
Eds. J.H. Dinitz and D.R. Stinson, J. Wiley, New York, 1992,
pp. 431--560.

\bibitem{Y}
C.H. Yang, On composition of four-symbol $\delta$-codes and
Hadamard matrices, Proc. Amer. Math. Soc. {\bf 107} (1989), 763--776.

\end{thebibliography}
\end{document}